\documentclass[12pt]{amsart}
\usepackage{amssymb,amsmath,amsthm,amsopn,amsfonts,graphics,xy,epsfig,picture,epic,version}

\usepackage{url}

\usepackage{amsmath}
\usepackage{amssymb}
\usepackage{amsfonts,color}
\usepackage{graphicx}
\usepackage{comment}
\usepackage{stmaryrd}

\usepackage{ulem}

\numberwithin{equation}{section}

\newcommand{\DD}{\mathbb{D}}
\newcommand{\NN}{\mathbb{N}}
\newcommand{\RR}{\mathbb{R}}
\newcommand{\TT}{\mathbb{T}}

\newcommand{\cD}{{\mathcal{D}}}

\newcommand{\bea}{\begin{align}}
\newcommand{\eea}{\end{align}}
\newcommand{\beqa}{\begin{align*}}
\newcommand{\eeqa}{\end{align*}}

\newcommand{\ve}{{\epsilon}}

\DeclareMathOperator{\Capa}{Cap}

\title[Random interpolation in Dirichlet spaces]{Random Interpolating Sequences in Dirichlet Spaces}

\author{N. ~Chalmoukis}
\author{A.~Hartmann}
\author{K.~Kellay}
\author{B.~D.~Wick}

\address{Dipartimento di Matematica, Universit\`a di Bologna, 40126, Bologna, Italy}\email{nikolaos.chalmoukis2@unibo.it}
\address{Univ. Bordeaux, CNRS, Bordeaux INP, IMB, UMR 5251,  F-33400, Talence, France}\email{Andreas.Hartmann@math.u-bordeaux.fr}
\address{Univ. Bordeaux, CNRS, Bordeaux INP, IMB, UMR 5251,  F-33400, Talence, France}\email{Karim.Kellay@math.u-bordeaux.fr}
\address{Department of Mathematics \& Statistics, Washington University -- St. Louis, One Brookings Drive, St. Louis, MO USA 63130--4899}
\email{wick@math.wustl.edu}



\thanks{The research of the first author is supported by the fellowship INDAM-DP-COFUND-2015 "INdAM Doctoral Programme in Mathematics and/or Applications cofunded by Marie Sklodowska-Curie Actions" \# 713485. The research of the second and third authors are  partially supported by the project ANR-18-CE40-0035 and by the Joint French-Russian Research Project PRC CNRS/RFBR 2017--2019.  The research of the last author is partially supported by National Science Foundation grants DMS \# 1800057 and DMS \#1560955.}

\theoremstyle{plain}
 \newtheorem{theorem}{Theorem}[section]
 \newtheorem*{theorem*}{Theorem}
 \newtheorem{lem}[theorem]{{Lemma}}
 \newtheorem{coro}[theorem]{{Corollary}}
 \newtheorem*{definition*}{{Definition}}

\theoremstyle{definition}

\begin{document}

\maketitle

\begin{abstract}
We discuss random interpolation in weighted  Dirichlet 
spaces $\cD_\alpha$, $0\leq \alpha\leq 1$. While conditions for deterministic interpolation 
in these spaces depend on capacities which are very hard to estimate in general, we show that random interpolation is driven 
by surprisingly simple distribution conditions. As a consequence, we obtain a breakpoint
at $\alpha=1/2$ in the behavior of these random interpolating sequences showing more precisely
that almost sure interpolating sequences for $\cD_\alpha$ 
are exactly the almost sure separated sequences when $0\le \alpha<1/2$
(which includes the Hardy space $H^2=\cD_0$), and they are exactly the almost sure zero sequences  for $\cD_\alpha$ when $1/2 \leq \alpha\le 1$ (which includes the classical Dirichlet space $\cD=\cD_1$). 
%
   \end{abstract}



\section{Introduction}

Understanding interpolating sequences is an important problem in complex analysis in one and
several variables. The characterization of when a sequence of points is an interpolating sequence finds many applications to different problems in signal theory, control
theory, operator theory, etc. In classical spaces like Hardy, Fock and Bergman spaces, interpolating
sequences are now well understood objects, at least in one variable \cite{G, S, Z}. In Dirichlet spaces, it turns out however that
getting an exploitable description of such interpolating sequences is a notoriously difficult problem related to capacities. Crucial work has been undertaken in the 90s by Bishop and Marshall-Sundberg (see more
precise indications below). However, while easier checkable sufficient conditions were given by Seip in the meantime (see \cite{S}), no real progress in the understanding of these sequences has
been made since those works. In such a situation, a probabilistic approach can lead to a new vision of these interpolating sequences. 
Note that besides the Hardy and Bergman spaces, the Dirichlet space, and its weighted companions, are  among the most prominent spaces of analytic functions on the unit disk. They appear naturally in problems on classical function theory, potential theory,  as well as in operator theory when one investigates for instance weighted shifts.

Here we consider random sequences of the following kind. Let
$\Lambda(\omega)=\{\lambda_n\}$ with $\lambda_n=\rho_n {\rm e}^{i\theta_n(\omega)}$ where  $\theta_n(\omega)$ is a sequence of independent random variables, all uniformly distributed on $[0,2\pi]$  (Steinhaus sequence), and  $\rho_n \in [0,1)$ is a sequence of {\it a priori} fixed radii. 
Depending on distribution conditions on $(\rho_n)$ as will be discussed below, we ask about the probability that $\Lambda(\omega)$ is interpolating for Dirichlet spaces $\mathcal{D}_\alpha$, $0\leq\alpha\leq1$. 
Recall that the weighted Dirichlet space $\mathcal{D}_\alpha$, $0\leq \alpha\leq 1$,  
 is  the space of all analytic function $f$ on the unit disc $\DD$ such that 
 $$
 \left\Vert f\right\Vert_{\alpha}^2:=|f(0)|^2+\int_\mathbb{D}|f'(z)|^2 (1-|z|^2)^{1-\alpha}{\rm d}A(z)<\infty,$$ 
where ${\rm d}A(z)={\rm d}x{\rm d}y/\pi$ stands for  the normalized area measure on $\mathbb{D}$ (we refer to \cite{EKMR} for Dirichlet spaces). For $f(z)=\sum_{n\ge 0}a_nz^n$, $z\in\DD$, the above expression is equivalent to $\sum_{n\ge 0}(1+n)^{\alpha}|a_n|^2$.
If $\alpha=0$, $\cD_0$ is the Hardy space $H^2$, and the classical Dirichlet space $\cD$ corresponds to $\alpha=1$. 

Recall that in a Hilbert space $H$ of functions analytic in the unit disk $\DD$ equipped with a reproducing kernel $k_\lambda$, i.e. $f(\lambda)=\langle f,k_{\lambda}\rangle_H$ for every
$\lambda\in\DD$ and $f\in H$ (a so-called reproducing kernel Hilbert space), a sequence $\Lambda$ of distinct points in $\DD$ is called (universal) interpolating if
$\{(f(\lambda)/\|k_\lambda\|_H)_{\lambda\in \Lambda}:f\in H\}=\ell^2$
(for the difference between interpolating and universal interpolating sequences see below).
Concerning the deterministic case of interpolation in the classical Dirichlet space $\cD$,
in unpublished work Bishop \cite{Bi} and, independently,
Marshall-Sundberg \cite{MS} characterized the interpolating sequences. 
The first published proof was given by B\o e \cite{Bo} who  provides a unifying scheme that applies to spaces that satisfy a certain property related to the so-called Pick property (see \cite{AgM,S}), and Dirichlet spaces fall in this category. For these spaces $\Lambda$ is a (universal) interpolating sequence if and only if $\Lambda$ is $H$-separated 
(i.e $\sup_{\lambda\neq\lambda^*\in\Lambda}|\langle k_\lambda/\|k_\lambda\|_H, k_{\lambda^*}/\| k_{\lambda^*} \|_H\rangle|<\delta_\Lambda<1$) and  
\[
\mu_{\Lambda}=\sum_{\lambda\in\Lambda}\delta_\lambda/\|k_\lambda\|^2_H
\] 
is a Carleson measure for $H$ (i.e,   $\int_\DD|f|^2{\rm d}\mu\leq C_\Lambda \|f\|_H^2$). 
Recently, Aleman, Hartz, McCarthy and Richter \cite{AHMR} have shown that this characterization 
remains valid in arbitrary reproducing kernel Hilbert spaces satisfying the complete Pick property.
Stegenga \cite{St} 
characterized Carleson measures for Dirichlet spaces, but his characterization is based
on capacities which are notoriously difficult to estimate for arbitrary unions of intervals.
There are other characterizations of Carleson measures in Dirichlet spaces, see \cite{ARS3, ARS1}, as well as \cite{ARSW,EKMR} and references therein, but which are not easily interpreted geometrically
for interpolating sequences. Finally, we mention related work
by Cohn \cite{Co} based on multipliers.

The interesting feature in connection with Stegenga's result is that switching to the random setting, it turns out that his condition for unions of intervals reduces to a one-box condition involving just one
interval, and for which the capacity can be estimated. This will be of central relevance
for our discussions.

We also would like to observe that generally, when the deterministic frame does not give a full answer to a problem, or if the deterministic conditions are not so easy to check, it is interesting to look at the random situation. In particular, it is interesting to ask for conditions ensuring that a sequence picked at random is or is not interpolating almost surely (i.e., which are in a sense ``generic
situations"?). In this context, it is also worth mentioning the huge existing literature around Gaussian analytic functions which investigates the zero distribution in classes of such functions  \cite{Peres}.
\\

The problems we would like to study in this paper are inspired by results by Cochran \cite{C} and Rudowicz \cite{R} who considered random interpolation in the Hardy space. 
Since interpolation in this space is characterized by separation (in the pseudohyperbolic metric) and by the Carleson measure condition (note that the Hardy space was the pioneering space with a kernel satisfying the complete Pick property), those authors were interested in a 0-1 law for separation, see \cite{C}, and a condition for being almost surely a Carleson measure \cite{R}, which led to a 0-1 law for interpolation.
It is thus natural to discuss separation, Carleson measure type conditions and interpolation in Dirichlet spaces.  

Concerning separation in Dirichlet spaces $\cD_\alpha$, $0<\alpha<1$, this turns out to be the same as in the Hardy spaces (see \cite[p.22]{S}), so that in that case Cochran's result perfectly characterizes the situation.
The separation in the classical Dirichlet space, however, is much more delicate than in the Hardy space. We establish here a 0-1-type law for separation in $\cD$. While our proof of this fact is inspired by Cochran's ideas, it requires a careful adaptation to the metric in that space.

Concerning Carleson measure type results in Dirichlet spaces, $0<\alpha<1$, 
it is instructive to first rediscuss the situation in the Hardy space. As it turns out we are
able to improve Rudowicz' result with a simplified proof based on generating functions and Markov's inequality instead of Rudowicz's proof using Tchebychev's inequality. This proof carries over to the Dirichlet situation and allows, together with Stegenga's characterization of Carleson measures,  to discuss the results on interpolation in $\cD_{\alpha}$.
A peculiar breakpoint in the behaviour of such interpolating sequences depending on the weight $\alpha$ is now appearing: for $0\le \alpha<1/2$, almost sure separation corresponds to almost sure interpolation, while for $1/2\le \alpha\le 1$, almost sure zero sequences correspond to almost sure interpolating sequences. Observe that for the critical value $\alpha=1/2$ the behaviour is the same as in the Dirichlet space. Commenting further on this breakpoint $\alpha=1/2$, we would like to mention work by Newman and Shapiro \cite{NS} who show that if $I$ is a non-constant singular inner function, then $I\notin \cD_{1/2}$.

Since zero sequences are of some importance as we have just seen, another central ingredient of our discussion is a rather immediate adaption of Bogdan's result on almost sure zero sequences in the Dirichlet space to the case of weighted Dirichlet spaces which we add for completeness 
in an annex.
\\

As usual, the definition of interpolating sequences is based on 
the reproducing kernel  of $\cD_\alpha$:
\begin{equation}\label{RepKern}
k_z(w)=\left\{\begin{array}{lll}
\displaystyle\frac{1}{\overline{z}w}\log\frac{1}{1-\overline{z}w}&\text{ if }& \alpha=1,\\
\displaystyle \frac{1}{(1-\bar zw)^{1-\alpha}},  &\text{ if }  &0\leq \alpha<1.
\end{array}
\right.
\end{equation}

We mention that contrary to the Hardy space situation,
it turns out that in certain spaces (e.g.\ the Dirichlet space)
there exist two notions of interpolation
depending on whether the restriction operator $R_{\Lambda}:H\longrightarrow \ell^2$, $R_\Lambda f=
(f(\lambda)/\|k_\lambda\|_H)_{\lambda\in \Lambda}$ takes values in $\ell^2$
or not. Notice that $\|R_{\Lambda}f\|_{\ell^2}^2=\int_{\DD}|f|^2d\mu_{\Lambda}$.

In $\cD_{\alpha}$, $0\le\alpha\le 1$,
a sequence $\Lambda$ of distinct points in $\DD$  said to be
\begin{itemize}
\item {\it a (simple) interpolating sequence}  if  
$R_{\Lambda}:f\to   (f(\lambda)/\|k_\lambda\|_\alpha)_{\lambda\in \Lambda}$ is onto $\ell^2$,
i.e.\ the interpolation problem $f(\lambda)=a_\lambda$ has a solution $f\in \cD_\alpha$ for every sequence $(a_{\lambda})$ with $(a_\lambda/\|k_\lambda\|_\alpha)_{\lambda\in \Lambda}\in \ell^2(\Lambda)$,
\item   a {\it universal interpolating} sequence   if  it is interpolating
and moreover $R_{\Lambda}$ is well defined from $\cD_{\alpha}$ into $\ell^2$.
\end{itemize}

Sequences  which are interpolating for the Dirichlet space but not universally interpolating were discovered by Bishop in \cite{Bi}, and they were further analyzed in \cite{ARS} and \cite{Ch}.
\\

In this paper we will exclusively consider universal interpolation. 
In particular, when speaking about interpolation we mean universal interpolation.\\

We will now discuss in details the results we have obtained in this paper.

\subsection{Back to the Hardy space}

As pointed out in the introduction, before considering the situation in the Dirichlet space, it seems appropriate to re-examine the situation in the Hardy space. Recall that Cochran established a 0-1 law for (pseudohyperbolic) separation (see Theorem \ref{Cochran} below) and Rudowicz showed that 
Cochran's condition for separation 
implies almost surely the Carleson measure condition. This implies that interpolation is characterized by the condition ensuring almost sure separation. In Dirichlet spaces, however, the situation is quite different. So, in order to get a better understanding of the situation we start stating an improvement of Rudowicz' results on random Carleson measures in the Hardy space which will help to better understand the case of Dirichlet spaces.

Recall that the measure ${\rm d}\mu_{\Lambda}=\sum_{\lambda\in \Lambda}(1-|\lambda|^2)\delta_{\lambda}$
is called a Carleson measure for $H^2$ if there is a constant $C$ such that for every interval $I\subset \TT$,
\[
 \mu_{\Lambda}(S_I)\le C|I|,
\]
where $|E|$ denotes normalized Lebesgue measure on $\TT$ of a measurable
set $E$, and
\[
 S_I=\big\{z=r{\rm e}^{it}\in \DD:{\rm e}^{it}\in I, 1-|I|\le r<1\big\}
\]
is the usual Carleson window (see \cite{G}). We will prove that  a weaker condition than Rudowicz' leads to Carleson measures almost
surely in the Hardy space. 
We first need to introduce a notation: let
$$A_n=\{ \lambda\in \Lambda (\omega)\text{ : } 1-2^{-n}\leq |\lambda|< 1-2^{-n-1}\},\quad n=0,1,2\ldots 
$$
be a dydadic annulus. Then we denote by $N_n$, the number of points
of $\Lambda$ contained in $A_n$: 
$$
 N_n=\# A_n.
$$
\begin{theorem}\label{Hardy1}
Let $\beta>1$ and suppose
\[
 \sum_{n\ge 1}2^{-n}N_n^{\beta}<+\infty.
\]
Then the measure ${\rm d}\mu_{\Lambda}$ is a Carleson measure almost surely in the Hardy space.
\end{theorem}
As a result, the Carleson measure condition alone is not sufficient to give a 0-1 law for interpolation in the Hardy space.

Note that Rudowicz \cite{R} showed that the above condition with $\beta=2$ is sufficient.
We will construct an example showing that it is not possible to replace $\beta$ 
by 1, so that (almost sure or not) zero sequences do not imply almost surely the 
Carleson measure condition. This makes the Hardy space a singular point in this respect 
within the
scale of weighted Dirichlet spaces, $0\le \alpha\le 1$.

\subsection{Interpolation in Dirichlet spaces $\cD_{\alpha}$, $0<\alpha<1$}


In this subsection we mention the results connected to interpolation: zero sequences, separation, and Carleson measures. 

\subsubsection{\bf Random zero sequences in Dirichlet spaces}

A central role in our interpolation results will be played by random zero sequences.
Indeed, for an interpolating sequence in the Dirichlet space it is necessary to be a
zero sequence (interpolation implies that there are functions vanishing on the whole sequence
except for one point $\lambda$, and multiplying this function by $(z-\lambda)$ yields a
nontrivial function in the Dirichlet space vanishing on the whole sequence).
We recall some results on random zero set in Dirichlet spaces. Carleson proved in \cite{Ca} that when
\begin{equation}\label{car}
\sum_{\lambda\in \Lambda} \|k_{\lambda}\|^{-2}_{\alpha}<\infty
\end{equation}
then  the Blaschke product $B$ associated to $\Lambda$  belongs to $\cD_\alpha$, $0<\alpha<1$ (for $\alpha=0$ this corresponds to the Blaschke condition for the Hardy space).  When  $\alpha=1$, 
Shapiro--Shields proved in \cite{SS} that the condition
\eqref{car}
is sufficient for $\{\lambda\}_{\lambda\in \Lambda} $ to be a  zero set for the classical Dirichlet space $\cD_1$, see also \cite[Theorem 1]{S}. Note that if $0\leq \alpha<1$ then  
$$ \sum_{\lambda\in \Lambda}\|k_{\lambda}\|^{-2}_{\alpha}\asymp\sum_{\lambda\in \Lambda}(1-|\lambda|)^{1-\alpha}\asymp \sum_n\ 2^{-(1-\alpha)n}N_n$$ and if $\alpha=1$  then $$ \sum_{\lambda\in \Lambda}\|k_{\lambda}\|^{-2}_{1}\asymp\sum_{\lambda\in \Lambda} |\log (1-|\lambda|)|^{-1}\asymp  \sum_n n^{-1}N_n.$$ On the other hand, it was proved by   Nagel--Rudin--Shapiro  in \cite{NRS}  that if $\{r_n\}\subset (0,1)$ does not satisfy \eqref{car}, then there is $\{\theta_n\}$ such that $\{r_ne^{i\theta_n}\}$ is not a zero set for $\cD_\alpha$.  Bogdan \cite[Theorem 2]{Bo} gives  a condition on the radii $|\lambda_n|$ for the sequence $\Lambda(\omega)$ to be almost surely zeros sequence for $\cD$:
\begin{equation}\label{ZeroD}
P(\Lambda(\omega) \text{ is a zero set for } \cD )=\left\{
\begin{array}{l}
1\\
0
\end{array}
\right.
\text{ if and only if } \sum_{n}n^{-1}N_n\left\{
\begin{array}{lll}
 <\infty\\
 =\infty.\\
\end{array}
\right. 
\end{equation}

Bogdan's arguments carry over to $\cD_{\alpha}$, $\alpha\in (0,1)$. For the sake of completeness,   we will prove in 
the annex, Section \ref{A}, the following result on almost sure zero sequences.

\begin{theorem}\label{alphazero}
Let $0\le \alpha<1$. Then
\begin{equation}\label{ZeroDalpha}
P(\Lambda(\omega) \text{ is a zero set for } \cD_\alpha )=\left\{
\begin{array}{l}
1\\
0
\end{array}
\right.
\text{ if and only if } \sum_{n}2^{-(1-\alpha)n}N_n \left\{
\begin{array}{lll}
 <\infty\\
 =\infty.\\
\end{array}
\right. 
\end{equation}
\end{theorem}



\subsubsection{\bf Interpolation in Dirichlet spaces $\cD_{\alpha}$, $0<\alpha<1$}  \label{SSIntDalpha}

As pointed out earlier, interpolation is intimately related with separation conditions
and Carleson measure type conditions. 
Recall that a sequence $\Lambda$ is called (pseudohyperbolically) separated if 
\[
\inf_{\lambda,\lambda^*\in \Lambda\atop \lambda\neq\lambda^*} \rho(\lambda,\lambda^*)=\inf_{\lambda,\lambda^*\in \Lambda\atop \lambda\neq\lambda^*} \frac{|\lambda-\lambda^*|}
 {|1-\overline{\lambda}\lambda^*|}\ge \delta_\Lambda>0.\\
\]

Since in Dirichlet spaces $\cD_{\alpha}$, $0\le \alpha <1$, the natural separation 
($\cD_\alpha$-separation) is indeed
pseudohyperbolic separation \cite[p.22]{S}, we recall Cochran's separation result on pseudohyperbolic
separation.

\begin{theorem}[Cochran]\label{Cochran}
A sequence $\Lambda(\omega)$ is almost surely (pseudohyperbolically) separated if and only if
\begin{equation} \label{CochranSep}
 \sum_n 2^{-n}N_n^2<+\infty.
\end{equation}
\end{theorem}

We should pause here to make a crucial observation. We have already mentioned that 
interpolating sequences are necessarily zero-sequences.
Also separation is another necessary condition for interpolation. Now the condition
for zero sequences \eqref{ZeroDalpha} depends on $\alpha$ while the separation condition
does not, and it follows that depending on $\alpha$, it is one condition or the other which is 
dominating. From \eqref{ZeroDalpha} and \eqref{CochranSep} it is not difficult to see that 
this breakpoint is exactly at $\alpha=1/2$ (for $\alpha=1/2$, \eqref{ZeroDalpha} still
implies \eqref{CochranSep}). This motivates already the necessary conditions
of our central Theorem \ref{CoroThm5} below. 
Another ingredient of that result comes from Carleson measures which have been characterized by Stegenga using capacitory conditions. For these we have the following result which is in the spirit of Theorem \ref{Hardy1} in the Hardy space.

\begin{theorem}\label{CarlAlpha1} 
Let $0<\alpha<1$. 

If 
\begin{equation}\label{ConditionDalpha1}
 \displaystyle\sum_n 2^{- (1-\alpha)n}N_{n}<\infty
\end{equation}
 then  $\mu_{\Lambda}=\sum_{\lambda\in\Lambda}(1-|\lambda|^2)^{1-\alpha}\delta_{\lambda}$ is almost surely a Carleson measure for $\cD_\alpha$.
\end{theorem}

Observe that contrarily to the Hardy space, where we had to pick $\beta>1$, here the exponent is exactly $(1-\alpha)$. We obtain a similar result in the classical Dirichlet space. So, in view of Theorem \ref{alphazero}, we can  deduce that almost sure zero sequences give rise  to almost sure Carleson measures in $\cD_{\alpha}$, $0<\alpha\le 1$ (which in particuliar includes the classical Dirichlet space).

We can now state our first main result.
\begin{theorem}\label{CoroThm5}
$\text{ }$

\begin{itemize}

\item[(i)] Let $0< \alpha< 1/2$.
Then
\[
 P\big({\Lambda}(\omega) \text{  is interpolating for $\cD_\alpha$} \big)=\left\{
\begin{array}{l}
1\\
0
\end{array}
\right.
\text{ if and only if } \sum_n 2^{- n}N_{n}^{2}\left\{
\begin{array}{lll}
 <\infty\\
 =\infty.\\
\end{array}
\right. 
\]

\item[(ii)] Let $1/2\leq\alpha<1$. Then
\[
 P\big({\Lambda}(\omega) \text{  is interpolating for $\cD_\alpha$} \big)=\left\{
\begin{array}{l}
1\\
0
\end{array}
\right.
\text{ if and only if } \sum_n 2^{- (1-\alpha)n}N_{n} \left\{
\begin{array}{lll}
 <\infty\\
 =\infty.\\
\end{array}
\right. 
\]

\end{itemize}
\end{theorem}

An interesting reformulation of the above results connects random interpolation with 
random zero sequences and random separated sequences as stated in the following
corollary.

\begin{coro}\label{zoub} The following statements hold:
\begin{enumerate}
\item  Let $0\leq \alpha<1/2$.  The sequence $\Lambda(\omega)$ is almost surely interpolating for $\cD_{\alpha}$ if and only if it is almost surely separated.
\item  Let $1/2\leq\alpha<1$.  The sequence $\Lambda(\omega)$ is almost surely interpolating for $\cD_{\alpha}$ if and only if it is almost surely a zero sequence.
\end{enumerate}
\end{coro}

Observe that for $\alpha=1/2$, the condition for being almost surely a zero sequence is strictly stronger than the condition for being almost surely separated so that in the limit case $\alpha=1/2$ it is indeed the almost sure zero condition that drives the situation.

\subsection{\bf Interpolation in the classical Dirichlet space}

For the classical Dirichlet space we will first establish a result on separation, and 
then use again the fact that in the random situation Stegengas capacitory condition on unions of intervals reduces to a single interval.

\subsubsection{\bf Separation in the  Dirichlet space} In the case $\alpha=1$, 
the separation is given in a different way.
 Let 
$$\rho_\cD(z,w)=\sqrt{
1-\frac{|k_w(z)|^2}{k_z(z)k_w(w)}}, \qquad z,w\in \DD.$$

A sequence $\Lambda$ is called $\cD$--separated  if
$$\inf_{\lambda,\lambda^*\in \Lambda\atop \lambda\neq\lambda^*} \rho_\cD(\lambda,\lambda^*)>\delta_\Lambda>0$$
for some $\delta_\Lambda<1$. This is equivalent to (see \cite[p.23]{S})
\begin{equation}\label{sepD}
\frac{(1-|\lambda|^2)(1-|\lambda^*|^2)}{|1-\overline{\lambda}\lambda^*|^2} \leq (1-|\lambda|^2)^{\delta_\Lambda^2},\qquad  \lambda,\lambda^*\in \Lambda.
\end{equation}

For separation in the Dirichlet space $\mathcal{D}$ we obtain the following 0-1 law.

\begin{theorem}\label{Dsepare}
$$
P(\Lambda(\omega)\text{ is  $\cD$--separated})=\left\{
\begin{array}{lllll}
1,& \text{ if  \, $\exists \gamma\in (1/2,1) $ such that } &\displaystyle \sum_n 2^{-\gamma n}N_{n}^{2} <\infty,\\
0,& \text{ if  \, $ \forall  \gamma \in (1/2,1)$ we have } &\displaystyle \sum_n 2^{-\gamma n }N_{n}^{2} =\infty.\\
\end{array}
\right.
$$
\end{theorem}

We observe that in both conditions we can replace the sum by a supremum (this amounts to 
replacing $\gamma$ by a slightly bigger or smaller value). The lower bound $1/2$ for $\gamma$ 
is not very important, since it is the behavior close to the value 1 which counts.

As we will see, this is a very mild condition in comparison to the almost sure zero- or Carleson-measure condition, and does not play a big role for the almost sure interpolation problem.

\subsubsection{\bf Interpolation in the Dirichlet space $\cD$}

Recall that Bogdan showed that $\Lambda(\omega)$ is almost surely a zero sequence
for $\cD$ if and only if $\sum_n n^{-1}N_n<+\infty$.
This motivates already the necessary part of the following complete characterization of almost surely {\it universal} interpolating sequences for $\cD$. The sufficiency comes essentially from the fact that the condition easily implies the (mild) separation condition, as well as the Carleson measure condition which will be shown in a similar fashion as in Theorem \ref{CarlAlpha1}.

\begin{theorem}\label{sintD}
\[
 P\big({\Lambda}(\omega) \text{  is universal interpolating for $\cD$} \big)=\left\{
\begin{array}{l}
1\\
0
\end{array}
\right.
\text{ if and only if } \sum_n n^{-1}N_{n}\left\{
\begin{array}{lll}
 <\infty\\
 =\infty.\\
\end{array}
\right. 
\]
\end{theorem}

We can reformulate the above result in the same spirit as Corollary \ref{zoub}
\begin{coro}
A sequence is almost surely interpolating for $ \cD $ if and only if it is almost surely a zero sequence
for $\cD$.
\end{coro}

\subsection{Organization of the paper}
This paper is organized as follows. In the next section we present the improved version of the
Rudowicz result concerning random Carleson measures in Hardy spaces which is the guideline for the corresponding result in the Dirichlet space. Indeed, this largely clarifies and simplifies not only the situation in the Hardy space, but also indicates the direction of investigation for the Dirichlet space. This will allow us to start Section 3 with the discussion of the Carleson measure result in $\cD_{\alpha}$, $0<\alpha<1$. We then solve completely the interpolation situation in these spaces.
The next section is devoted to the
0-1 law on separation in the classical Dirichlet space. This requires a subtle
adaption of the Cochran discussion in the Hardy space to the much more intricate
geometry in the Dirichlet space. 
The proofs of the results
on interpolating sequences in the classical Dirichlet space are contained in Section 5.
 Actually, as in the Hardy space,
the core of the proof being probabilistic, we are able to get rid of analytic functions.
In the final Section 6, we give some indications to the 0-1 law on zero-sequences in 
weighted Dirichlet spaces based on Bogdan's proof in the classical Dirichlet space.
\\

A word on notation. Suppose
$A$ and $B$ are strictly positive expressions.
We will write $A\lesssim B$ meaning that $A\le cB$ for some positive constant $c$ not depending on the parameters behind $A$ and $B$. By $A\simeq B$ we mean $A\lesssim B$ and $B\lesssim A$. We further use the notation $A\sim B$ provided the quotient $A/B\to 1$ when passing to the suitable limit.



\section{Carleson condition in the Hardy space}\label{Hardy}

Before considering Carleson measure conditions in the Dirichlet space, we will
discuss the situation in the Hardy space, in particular we will prove here
Theorem \ref{Hardy1}.
We will also construct an example showing that it is not possible to choose $\beta=1$
in the statement of Theorem \ref{Hardy1}.

\subsection{Proof of Theorem \ref{Hardy1}} 

We start introducing some notation. Let 
$$I_{n,k}=\{{\rm e}^{2\pi it}:t\in [k2^{-n},(k+1)2^{-n})\}\quad n\in\NN, \quad k=0,1,\ldots, 2^n-1$$
be dyadic intervals and 
$S_{n,k}=S_{I_{n,k}}$ the associated Carleson window. In order to check the Carleson
measure condition for a positive Borel measure $\mu$ on $\DD$ it is clearly
sufficient to check the Carleson measure condition for windows $S_{n,k}$:
\[
 \mu(S_{n,k})\le C|I_{n,k}|= C2^{-n},
\]
for some fixed $C>0$ and every $n\in\NN$, $k=0,\ldots,2^n-1$.
Given $n$, $k$ and $m\ge k$ let $X_{n,m,k}$ be the number of points of $\Lambda$
contained in $S_{n,k}\cap A_m$ (we stratify the Carleson window $S_{n,k}$ into
a disjoint union of layers $S_{n,k}\cap A_m$). Since $A_m$ contains $N_m$ points and
the (normalized) length of $I_{n,k}$ is $2^{-n}$, we have $X_{n,m,k}\sim 
B(2^{-n},N_m)$ (binomial law). In order to show that $d\mu_{\Lambda}$ is almost surely a Carleson
measure we thus have to prove the existence of $C$ such that
\[ 
 \mu_{\Lambda}(S_{n,k})=\sum_{m\ge n}2^{-m}X_{n,m,k}\le C 2^{-n}
\] 
almost surely, in other words we have to prove
\[
 \sup_{n,k}2^n\sum_{m\ge n}2^{-m}X_{n,m,k}\le C
\]
almost surely (in $\omega$). The estimate above had already been investigated by
Rudowicz \cite{R}. Here we will proceed in a different way with respect to Rudowicz' argument to obtain
an improved version of his result and which allows to better understand the Dirichlet space
situation.

\begin{proof}[Proof of Theorem \ref{Hardy1}]
In view of our preliminary remarks, we need to look at the random variable
\[
 Y_{n,k}=2^n\sum_{m=n}^{+\infty}2^{-m}X_{n,m,k},
\]
where, as said above, $X_{n,m,k}\sim B(2^{-n},N_m)$.
Hence, saying that $Y_{n,k}\ge A$ means that there are Carleson windows for which the
Carleson measure constant is at least $A$. 
 Also denote by $ G_{Y_{n,k}} $ the probability generating function of the random variable $ Y_{n,k} $, i.e. $ G_{Y_{n,k}}(s)=\mathbb{E}(s^{Y_{n,k}}) $. It is well known that for a random variable $ X $ which follows a binomial distribution with parameters $ p,N $ we have that $ G_X(s)=(1-p+ps)^N $.
	
	By the hypothesis, for $ n $ sufficiently large, $ N_n\leq 2^{(1-\ve)n}, \ve = 1-1/\beta $. Introduce now two parameters $ A,s>0 $ to be specified later. By Markov's inequality we have that
	\begin{align*}
	\log P(Y_{n,k} \geq A ) & =  \log P (s^{Y_{n,k}} \geq s^A) \\ 
	& \leq \log \big(\frac{1}{s^A}  G_{Y_{n,k}}(s) \big ) \\
	& = \sum_{m \geq n}N_m\log(1-2^{-n}+2^{-n}s^{2^{n-m}})-A\log(s) \\
	& \leq 2^{-n} \sum_{m \geq n}N_m(s^{2^{n-m}}-1)-A\log(s) \\
	& = \sum_{m\geq 0} N_{n+m}2^{-(n+m)}2^{m}(s^{2^{-m}}-1)-A\log(s). \\
	\end{align*}
	At this point notice that $ x(a^{1/x}-1)\leq a , $ for all $ x\geq 1, a>0 $, which together with the hypothesis on $ N_n $ gives 
	\[
	\log P(Y_{n,k} \geq A ) \leq \sum_{m\geq 0}2^{-\ve(n+m)}s-A\log(s)= \frac{2^\ve}{2^\ve-1}s2^{-\ve n}- A \log(s).
	\]
	Now set $ s=2^{\frac{\ve n}{2}} , A= \frac{4}{\ve }$ in the last inequality to get
	\[
	\log P(Y_{n,k} \geq A ) \leq \frac{2^\ve}{2^\ve-1}2^{-\frac{\ve n}{2}}-2n\log(2). 
	\]
	Hence, $ P(Y_{n,k} \geq \frac{4}{\ve} ) \leq C(\ve) 2^{-2n} $.

In view of an application of the Borel-Cantelli Lemma we compute
\[
\sum_{n\ge 0}\sum_{k=1}^{2^n}P(Y_{n,k}\ge A)
 \le C(\ve)\sum_{n\ge 0}2^n\times 2^{-2n}<\infty.
\] 

Hence, by the Borel-Cantelli Lemma, the event $Y_{n,k}\ge A$ can happen for at most
a finite number of indices $(n,k)$ almost surely. In particular the Carleson measure constant of 
$d\mu_{\Lambda}$ is almost surely at most $A$ except for a finite number of 
Carleson windows.
\end{proof}

\subsection{An example}

Here we construct an example showing that we cannot choose $\beta=1$ in the theorem.
For $\gamma>1$, let 
\[
 N_n=\frac{2^n}{n^{\gamma}},\quad n\ge 1.
\]
Clearly $\sum_{n\ge 1}2^{-n}N_n<+\infty$. For $n\in \NN^*$ and $k=0,1,\ldots, 2^n -1$
define the event
\[
 A^N_{n,k}=\bigcup_{m=n}^{n+N}(X_{n,k,m}\ge 2^{m-n}).
\]
Since $X_{n,k,m}$ are mutually independent for fixed $n,k$, we have
\[
 P(A^N_{n,k})=\prod_{m=n}^{n+N}P(X_{n,k,m}\ge 2^{m-n}).
\]
The following well known elementary lemma 
will be very useful 
(it is essentially approximation of the binomial law by the Poisson law). We refer for instance to
\cite{Bil} for the material on probability theory -- essentially elementary -- used in this paper.

\begin{lem}\label{bino} If $X$ is a binomial random variable with parameters $p,N,$ then for every $s=0,1,2,\ldots$,
	$$\lim_{N\to \infty\atop pN\to 0}\frac{P(X=s)}{(pN)^s}=\lim_{N\to \infty\atop pN\to 0}\frac{P(X\geq s)}{(pN)^s}=\frac{1}{s!}.
	$$
\end{lem}

Note that in view of the Lemma we can replace in $A_{n,k}^N$ the condition 
$X_{n,k,m}\ge 2^{m-n}$ by $X_{n,k,m}=2^{n-m}$.
Since in our situation $N=2^m/m^{\gamma}\to \infty$ and $pN=N_m/2^n
\le 2^{m-n}/m^{\gamma}\to 0$ ($n\le m\le n+N$, $N$ fixed), we get
\begin{eqnarray}\label{PAnk}
 P(A^N_{n,k})&=&\prod_{m=n}^{n+N}P(X_{n,k,m}= 2^{m-n})
 \simeq \prod_{m=n}^{n+N} \left(\frac{1}{(2^{m-n})!} 
 \left(\frac{N_m}{2^n}\right)^{2^{m-n}}\right)\nonumber\\
 &=&\prod_{m=n}^{n+N} \left(\frac{1}{(2^{m-n})!} 
 \left(\frac{2^{m-n}}{m^{\gamma}}\right)^{2^{m-n}}\right)\nonumber\\
 &=&\prod_{m=0}^{N} \left(\frac{1}{(2^{m})!} 
 \left(\frac{2^{m}}{(m+n)^{\gamma}}\right)^{2^{m}}\right)\nonumber\\
&\simeq& \frac{1}{n^{2^{N+1}\gamma}}.
\end{eqnarray}
The crucial observation here is that for every fixed $N$ this probability goes polynomially
to zero. In particular 
\[
 \sum_{n\ge 0}\sum_{k=0}^{2^n-1}P(A^N_{n,k})\sim\sum_{n\ge 0}\frac{2^n}
 {n^{2^{N+1}\gamma}}=+\infty.
\]
We need to apply a reverse version of the Borel-Cantelli lemma. This works for 
events which are independent. However this is not the case for $A^N_{n,k}$ and
$A^N_{i,j}$ when the corresponding dyadic annuli meet. Since $A^N_{n,k}$ intersects the annuli
$A_l$, $n\le l\le n+N$ and $A^N_{i,j}$ intersects $A_l$, $i\le l\le i+N$, the events $A^N_{n,k}$
and $A^N_{i,j}$ are dependent when $[ n,n+N]\cap
[ i,i+N]\neq \emptyset$, i.e.\ when $|n-i|\le N+1$. 
We will appeal to a more general version 
of the Borel-Cantelli Lemma, see Lemma \ref{BoCanV2}. This requires that
\[
\liminf_{M\to+\infty}\frac{\sum_{i,j,k,n\le M}P(A_{i,j}^N\cap A_{n,k}^N)}
 {\Big(\sum_{k,n\le M}P(A^N_{n,k})\Big)^2}\le 1.
\] 
The difference here is of course made by the  elements with $|n-i|\le N+1$ since for 
$|n-i|>N+1$, we have $ P(A_{i,j}^N\cap A_{n,k}^N)=P(A_{i,j}^N)\times P(A_{n,k}^N)$.
In particular
\begin{eqnarray*}
\lefteqn{\sum_{i,j,k,n\le M}P(A_{i,j}^N\cap A_{n,k}^N)}\\
&&=\Big(\sum_{k,n\le M}P(A^N_{n,k})\Big)^2
 +\sum_{\stackrel{|n-i|\le N+1}{i,j,k,n\le M}}P(A_{i,j}^N\cap A_{n,k}^N)
 -\sum_{\stackrel{|n-i|\le N+1}{i,j,k,n\le M}}P(A_{i,j}^N)\times P( A_{n,k}^N). 
\end{eqnarray*}
For fixed $N$ we can assume $i$ and $n$ big so that $N\ll N_n, N_i$.
Note that $P(A_{i,j}^N\cap A_{n,k}^N)=P(A_{i,j}^N)\times P(A_{n,k}^N|{A_{i,j}^N})$, and
we have to estimate $P(A_{n,k}^N|{A_{i,j}^N})$. The idea is to observe that the conditional
probability essentially behaves like the unconditional one, i.e.\ the knowledge of
$A_{i,j}^N$ does not interfer too much on the probability of $A_{n,k}^N$.

More precisely, if $|n-i|\le N+1$, the fact that $A_{i,j}^N$ has occured reduces the number of
points available for $A_{n,k}^N$ (in the annulus $A_m$) 
by at most $2^N$ which we can again assume
neglectible with respect to $N_m$, $n\le m\le n+N$. 

One has to pay a little bit attention here. 
A priori, it could happen that $S_{i,j}\subset S_{n,k}$ or $S_{n,k}\subset S_{i,j}$ (in these cases
the knowledge of $A_{i,j}^N$ has some indicidence to that of $A_{n,k}^N$ or vice versa), but this
requires $|j-k|\le 2^N$ which is uniformly bounded. The corresponding sum of 
probabilities is thus also uniformly bounded and dividing by the square of eventually
divergent partial sums
makes these terms neglectible. So we can assume that this pathological situation does not 
occur.

Then, in order to estimate the probability of 
$A_{n,k}^N$, under the condition $A_{i,j}^N$, we can run a similar computation as in 
\eqref{PAnk} with $N_m$ replaced by $N_m-N=N_m(1-N/N_m)$ (for those $m$ for which the 
annulus $A_m$ lies in both Carleson boxes;
for the others we keep $N_m$), which yields a 
comparable probability:
\[
 P(A_{i,j}^N\cap A_{n,k}^N)=(1+\varepsilon_{n,k,i,j}) P(A_{i,j}^N)\times P( A_{n,k}^N)
\]
where $\varepsilon_{n,k,i,j}\to 0$ when $i,j,n,k$ get big.
Then
\begin{eqnarray*}
\lefteqn{\sum_{\stackrel{|n-i|\le N+1}{i,j,k,n\le M}}P(A_{i,j}^N\cap A_{n,k}^N)
 -\sum_{\stackrel{|n-i|\le N+1}{i,j,k,n\le M}}P(A_{i,j}^N)\times P( A_{n,k}^N)}\\
&&=\sum_{n,k\le M}P(A_{n,k}^M)(1-P(A_{n,k}^M)
+
\sum_{\stackrel{0<|n-i|\le N+1}{i,j,k,n\le M}}\varepsilon_{n,k,i,j}P(A_{i,j}^N)\times P( A_{n,k}^N),
\end{eqnarray*}
which, when dividing through $\Big(\sum_{k,n\le M}P(A^N_{n,k})\Big)^2\to
+\infty$, $M\to\infty$, goes to 0.

From Lemma \ref{BoCanV2} we conclude that for every $N$ there exists infinitely many
$n,k$ such that $\mu_{\Lambda}(S_{n,k})\ge N+1$, which concludes the example.


\section{Interpolation in $\cD_{\alpha}$, $0<\alpha <1$.}
\label{PfThm5}

We start this section with the proof
of Theorem
\ref{CarlAlpha1} which we recall here for convenience.
\begin{theorem*}
Let $0<\alpha<1$. 
If 
\begin{equation}
 \displaystyle\sum_n 2^{- (1-\alpha)n}N_{n}<\infty
\end{equation}
 then  $\mu_{\Lambda}$ is almost surely a Carleson measure for 
$\cD_\alpha$.
\end{theorem*}

Recall also that
\[
  \mu_{\Lambda}=\sum_{\lambda\in\Lambda}\frac{1}{\|k_{\lambda}\|_{\alpha}^2}
 \delta_{\lambda},
\]
where
\[
 \|k_{\lambda}\|_{\alpha}^2=\left\{
 \begin{array}{ll}
 \log\displaystyle\frac{1}{1-|\lambda|^2},& \alpha=1,\\
 \displaystyle\frac{1}{(1-|\lambda|^2)^{1-\alpha}},& 0\le \alpha<1.
 \end{array}
 \right.
\] 
The case $\alpha=1$ will be useful later in the study of the classical Dirichlet space $\cD$.
\\

It shall be observed that the proof presented below does not work for the Hardy space case $\alpha=0$, for
which we have seen that it is possible to construct sequences $(r_n)$ satisfying the Blaschke 
condition, but the associated sequences $\Lambda(\omega)$ are not almost surely interpolating. 
\\

We will need Stegenga's characterization of Carleson measures for Dirichlet spaces which
involves capacities. In connection with this result we recall the the following three facts.
Once these facts collected, the proof is essentially the same as in the Hardy space.

By $\Capa_{\alpha}$ we will mean logarithmic capacity when $\alpha=1$ and just 
$\alpha$-capacity when $\alpha\in (0,1)$ (we refer the reader to e.g.\ \cite[p.19]{S}, 
\cite{EKMR} for more information on capacities). The first fact we would like to recall are the 
following known estimates (see \cite[p.19]{S}) 
\begin{eqnarray}\label{CapEstim}
\Capa_{\alpha}(I)\simeq \left\{
 \begin{array}{ll}
  |I|^{1-\alpha},& 0<\alpha<1,\\
 \left(\log\displaystyle\frac{e}{|I|}\right)^{-1},& \alpha=1,
 \end{array}
\right.
\end{eqnarray}
where $I$ is an interval.

The second fact is Stegenga's result (see e.g.\ \cite[p.19]{S}).

\begin{theorem}\label{StegThm}
A nonnegative Borel measure $\mu$ on $\DD$ is a Carleson measure for $\mathcal{D}_{\alpha}$, $0<\alpha\le 1$, if and only if there exists a constant $c>0$ such that
\begin{eqnarray}\label{CharSteg}
 \sum_{j=1}^n\mu(S_{I_j})\le c\Capa_{\alpha}\left(\bigcup_{j=1}^nI_j\right),
\end{eqnarray}
for each finite collection of disjoint subarcs $I_1, I_2,\ldots, I_n$ of the unit circle,
and arbitrary $n$.
\end{theorem}

The last fact is the following observation. There exists a universal constant $c$ such that
for every finite collection $I_1$, $I_2$, ..., $I_n$ of subsarcs of $\TT$, and $I$ an arc of 
$\TT$, with $|I|=\sum_{j=1}^n|I_j|$, we have
\begin{eqnarray}\label{Spread}
 \Capa_{\alpha}(I)\le c \Capa_{\alpha}\left(\bigcup_{j=1}^nI_j\right).
\end{eqnarray}
Concerning \eqref{Spread},
we refer to \cite[Theorem 2.4.5]{EKMR} for the case $\alpha=1$. The general case 
$\alpha\in (0,1)$ is shown exactly in the same way as for $\alpha=1$.

\begin{proof}[Proof of Theorem \ref{CarlAlpha1}]
We can now essentially repeat the argument from the Hardy space case. In view
of \eqref{Spread}, in order to apply Theorem \ref{StegThm}, it is sufficient to show 
that the condition of Theorem \ref{CarlAlpha1} implies almost surely
\[
 \sum_{j=1}^n\mu(S_{I_j})\le c\Capa_{\alpha}(I),
\]
where $I$ is an interval of length $\sum_{j=1}^n|I_j|$. 
Since the distribution of points in $S_J$ for an arc $J$ does not depend on its position on
$\TT$, we can assume that $I_j$ are adjacent intervals. Then, setting 
$I=\bigcup_{j=1}^n I_j$ we have
\[
 \bigcup_{j=1}^n S_{I_j}\subset S_I,
\]
and for the application of Stegenga's theorem it is enough to show that almost surely
\begin{eqnarray}\label{1BoxSteg}
 \mu(S_I)\le c\Capa_{\alpha}(I).
\end{eqnarray}
In view of our discussion on $\mathcal{D}$ we should emphasize that the above discussion
is true for every $\alpha\in (0,1]$.\\

Consider now the case $\alpha\in (0,1)$. Then \eqref{1BoxSteg} becomes
\[
 \sum_{\lambda\in S_I}(1-|\lambda|^2)^{1-\alpha}\le c|I|^{1-\alpha}.
\]
As in the Hardy space, it is enough to discuss this inequality for dyadic intervals 
$I=I_{n,k}$, and covering again $S_{n,k}$ by dyadic arcs we obtain the 
following random variable
\[
 Y_{n,k}^{\alpha}=2^{(1-\alpha)n}\sum_{m\ge n}2^{-(1-\alpha)m}X_{n,m,k}.
\]

We now get using again $ x(a^{1/x}-1)\leq a$ for $x\ge 1$, $a>0$, 
	\begin{align*}
	\log P(Y^{\alpha}_{n,k} \geq A ) & =  \log P (s^{Y_{n,k}} \geq s^A) 
	 \leq \log \big(\frac{1}{s^A}  G_{Y_{n,k}}(s) \big ) \\
	& = \sum_{m \geq n}N_m\log(1-2^{-n}+2^{-n}s^{2^{(1-\alpha)(n-m)}})-A\log(s) \\
	& \leq 2^{- n} \sum_{m \geq n}N_m(s^{2^{(1-\alpha)(n-m)}}-1)-A\log(s) \\
	& \leq 2^{-\alpha n} \sum_{m \geq n}N_m2^{-(1-\alpha)m}
 \underbrace{2^{(1-\alpha)(m-n)}(s^{2^{(1-\alpha)(n-m)}}-1)}_{\le s}-A\log(s) \sout{.}\\
 &\leq 2^{-\alpha n}sc_n- A \log(s),
	\end{align*}
where $c_n=\sum_{m\ge n}2^{-(1-\alpha)m}N_m$ is the value of the remainder sum which tends to zero, and which we thus
can assume less than 1.
	Now set $ s=2^{\alpha n} , A= \frac{4}{\alpha}$ in the last inequality to get
	\[
	\log P(Y_{n,k} \geq A ) \leq 1-2n\log(2). 
	\]
	Hence, $ P(Y_{n,k} \geq \frac{4}{\alpha} ) \leq C(\alpha) 2^{-2n} $,
and we conclude as in the Hardy space case.
\end{proof}

We are now in a position to prove Theorem \ref{CoroThm5}.

\begin{proof}[Proof of Theorem \ref{CoroThm5}]
(i) Let  $0<\alpha<1/2$. 

If $\Lambda$ is interpolating almost surely, then it is separated almost surely, which 
implies $\sum_n 2^{-n}N_n^2<+\infty$.

If $\sum_n 2^{-n}N_n^2<+\infty$, then $\Lambda$ is almost surely separated. 
Moreover, the condition implies that $N_n\le c2^{n/2}$ for some constant $c>0$, 
and hence 
\[
\sum_{n} 2^{-(1-\alpha)}N_n\le \sum_{n} 2^{-(1/2-\alpha)}<+\infty.
\]
Theorem \ref{CarlAlpha1} implies that $\mu_{\Lambda}$ is almost surely a Carleson
measure so that $\Lambda$ is almost surely interpolating.

The remaining cases of a divergent sum and interpolation with zero probability follows
from the above and the Kolomogorov 0-1 law.
\\

(ii) Consider the case $1/2\le \alpha<1$.

If $\Lambda$ is interpolating almost surely, then it is a zero sequence almost surely, which
implies $\sum_n 2^{- (1-\alpha)n}N_{n}<+\infty$ by Theorem \ref{alphazero}.

Suppose $\sum_n 2^{- (1-\alpha)n}N_{n}<+\infty$. 
Again by Theorem \ref{CarlAlpha1} we that $\mu_{\Lambda}$ is almost surely a Carleson
measure.

The condition clearly implies also that $\sum_{n}2^{-n}N_n^{2}<+\infty$, which further 
yields that the sequence is almost surely separated. Hence it is almost surely an
interpolating sequence.

Agian, the remaining cases of a divergent sum and interpolation with zero probability follows
from the above and the Kolomogorov 0-1 law.

\end{proof}


\section{Separated random sequences for the Dirichlet space}

We will now prove the separation result in $\cD$.

\begin{proof}[Proof of Theorem \ref{Dsepare}]
{\it Separation with probability 0}. Assume that for all $ \gamma \in (1/2,1)$ we have 
$\sup_k 2^{-\gamma k }N_{k}^{2} =\infty$.
As it turns out, under the condition of the Theorem, separation already fails in dyadic annuli (without 
taking into account radial Dirichlet separation).

Assume  now that $\gamma_l\to 1 $ as $l\to \infty$ and $\sup_k2^{-\gamma_lk }N_{k}^{2} =\infty$ for every $l$. For each $k=1,2\ldots$, let $I_k=[1-2^{-k+1},1-2^{-k})$.
Define
$$\Omega_k^{(l)}=\{\omega\text{ : } \exists (i,j), \; i\neq j \text{ with } \rho_i,\rho_j\in I_k \text{ and } |\theta_i(\omega)-\theta_j(\omega)|\leq \pi 2^{-\gamma_l k}\}.$$

In view of \eqref{sepD}, if $\omega\in \Omega_k^{(l)}$,
this means that in the dyadic annulus $A_k$ there are at least two points
close in the Dirichlet metric. To be more precise, if $\omega\in \Omega_k^{(l)}$, then
there is a pair  of distinct point $\lambda_i(\omega)$ and $\lambda_j(\omega)$ such that $|\lambda_i|,|\lambda_j|\in I_k$ and $|\arg\lambda_i(\omega)-\arg\lambda_j(\omega)|\leq \pi 2^{-\gamma_l k}$. Hence
$$
\frac{(1-|\lambda_i|^2)(1-|\lambda_j|^2)}{|1-\overline{\lambda_i}\lambda_j|^2} \geq c
\frac{2^{-2k}}{2^{-2k}+ \pi 2^{-2\gamma_l k}},
$$
where the constant $c$ is an absolute constant. Hence
$$
\frac{(1-|\lambda_i|^2)(1-|\lambda_j|^2)}{|1-\overline{\lambda_i}\lambda_j|^2} 
\ge c\frac{1}{1+\pi 2^{k(1-\gamma_l)}}
 \ge c' 2^{-k(1-\gamma_l)}\ge c''(1-|\lambda_i|)^{1-\gamma_l}.
$$
Absorbing $c''$ into a suitable change of the power $\delta_l^2:=1-\gamma_l$ into 
${\delta_l'}^2$ (which can be taken
by choosing for instance $2\delta_l>\delta_l'>\delta_l$ provided $k$ is large enough), then
by \eqref{sepD}
$$
 \rho_\cD (\lambda_i(\omega),\lambda_j(\omega))\le \delta'_l.
$$
Our aim is thus to show that for every $l\in\NN$, we can find almost surely $\lambda_i(\omega)
\neq \lambda_j(\omega)$ such that 
$\rho_\cD (\lambda_i(\omega),\lambda_j(\omega))\le \delta'_l$, i.e. $P(\Omega_k^{(l)})=1$. (Note that $\delta'_l\to 0$ when $l\to+\infty$.)
\\

Let us define a set $ E:=\{j:2^{-\gamma_lj-1}N_j\le 1\}$. 
Observe that when $k\notin E$, 
then at least two points are closer than $\pi 2^{-\gamma_l k}$ (this is completely deterministic), so that in that case $P(\Omega_k^{(l)})=1$. Hence if $E\subsetneqq \NN$, then we are done.

Consider now the case $E=\NN$, and let $k\in E=\NN$.
We will use the Lemma on the probability of an uncrowded road \cite[p. 740]{C}, which states
$$P(\Omega_k^{(l)})=1-(1-N_k 2^{-\gamma_l k-1})^{N_k-1}$$
(since $E=\NN$ this is well defined).

We can assume that $N_k\ge 2$ (since obviously
$\sum_{k:N_k<2}2^{-\gamma_l k }N_k
<\infty$). In particular $N_k^2/2\le N_k(N_k-1)\le N_k^2$.
Since $\log(1-x)\leq -x$, we get
\begin{eqnarray*}
\sum_{k:N_k\ge 2} (N_k-1)\log (1- N_{k} 2^{-\gamma_l k-1})
 &\le& -\sum_{k:M_k\ge 2} (N_k-1) N_{k} 2^{-\gamma_l k-1} \\
 &\le& -\frac{1}{2}\sum_{k:N_k\ge 2} N_k^2 2^{-\gamma_l k-1} \\
 &=&-\infty
\end{eqnarray*}
by assumption.
Hence, taking exponentials in the previous estimate,
$$\prod_{k\in E,N_k\ge 2} (1- N_{k} 2^{-\gamma_l k-1})^{N_k-1}=0,$$ 
which implies, by results on convergence on infinite products, that
$$\sum_kP(\Omega_k^{(l)})=\infty.$$
Since the events $\Omega_k^{(l)}$ are independent, by  the Borel--Cantelli Lemma, 
$$P(\limsup \Omega_k^{(l)})=1,$$ where 
$$\limsup \Omega_k^{(l)}=\bigcap_{n\geq 1}\bigcup_{k\geq n} \Omega_k^{(l)}=
\{\omega\text{  : }\omega\in \Omega_k^{(l)} \text{ for infinitely many } k\}.$$
In particular, since the probability of being in infinitely many $\Omega_k^{(l)}$ is one,
there is at least one $\Omega_k^{(l)}$ which happens with probability one. So that
again $P(\Omega_k^{(l)})=1$.
\\

As a result, the probability that the sequence is $\delta_l'$-separated in the Dirichlet
metric is zero for every $l$. Since $\delta_l'\to 0$ when $l\to+\infty$, we deduce that

$$P(\omega\text{ : }\{\lambda(\omega) \} \text{ is  separated for }\cD)=0.$$
\\

{\it Separation with probability 1.}
Now we 
assume that $\sum_k 2^{-\gamma k}N_k^2<+\infty$  for some $\gamma\in (1/2,1)$.
Let us begin defining a neighborhood in the Dirichlet
metric. For that, fix $\eta>1$ and $\alpha\in (0,1)$. Given $\lambda\in \Lambda$,
so that for some $k$, $\lambda\in A_k$.
Consider
\[
 T_{\lambda}^{\eta,\alpha}=\{z=r{\rm e}^{it}: (1-|\lambda|)^{\eta}\le 1-r\le 
  (1-|\lambda|)^{1/\eta}, |\theta-t|\le (1-r)^{\alpha}\}.
\]
Figure \ref{fig1} represents the situation.

\begin{figure}[t]                                                       %
 \includegraphics{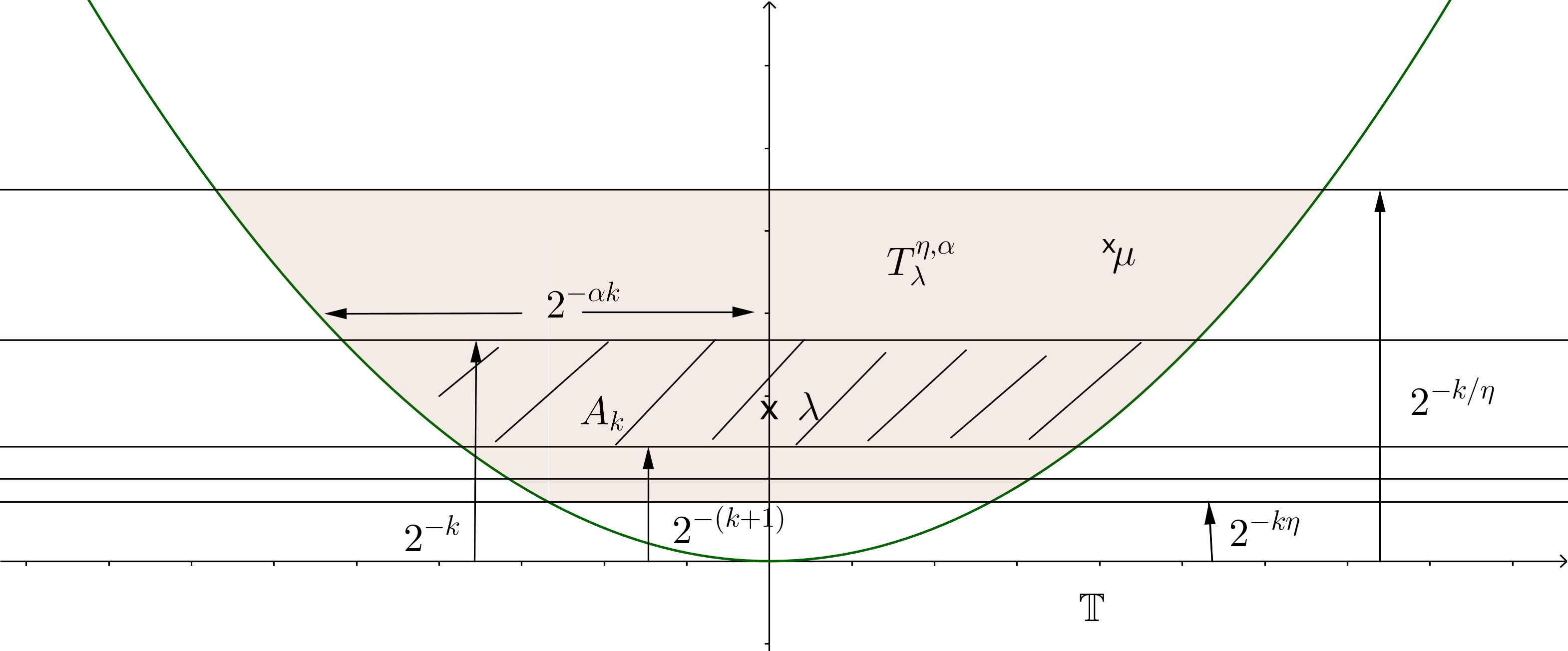} 
   \caption{Dirichlet neighborhood.}\label{fig1}                                     %
\end{figure}

Our aim is to prove that under the condition $\sum_k 2^{-\gamma k}N_k^2<+\infty$, 
there exists
$\eta>1$ and $\alpha\in (0,1)$ such that $T_{\lambda}^{\eta,\alpha}$ does not contain any
other point of $\Lambda$ except $\lambda$, and this is true for every $\lambda\in\Lambda$
with probability one. For this we need to estimate
\[
 P(T_{\lambda}^{\eta,\alpha}\cap\Lambda=\{\lambda\}).
\]

Let us cover
\[
 T_{\lambda}^{\eta,\alpha}= \bigcup_{j=k/\eta}^{\eta k}(T_{\lambda}^{\eta,\alpha}\cap A_j),
\]
and we need that for every $j\in [k/\eta,\eta k]\setminus\{k\}$,
$(T_{\lambda}^{\eta,\alpha}\cap A_j)\cap\Lambda=\emptyset$ and
$(T_{\lambda}^{\eta,\alpha}\cap A_k)\cap\Lambda=\{\lambda\}$.
Note that $X_j=\#(T_{\lambda}^{\eta,\alpha}\cap A_j\cap\Lambda)\sim B(N_j,2^{-\alpha j})$,
$j\ne k$, and $X_k\sim B(N_k-1,2^{-\alpha k})$ (since we do not count $\lambda$ in the
latter case). Hence, since the arguments of the points are independent, we have
\begin{align*}
 P(T_{\lambda}^{\eta,\alpha}\cap\Lambda=\{\lambda\})
 &= P\Big(\big(\bigcap_{j=k/\eta,j\neq k}^{\eta k}(X_j=0)\big)\cap (X_k=1)\Big)\\
 &=\prod_{j=k/\eta,j\neq k}^{j=\eta k} \Big(P(X_j=0)\Big)\times P(X_k=1).
\end{align*}
From the binomial law we have $P(X_j=0)=(1-2^{-\alpha j})^{N_j}$, for $j\in [k/\eta,\eta k]\setminus\{k\}$.
Also,  assuming $0<\gamma<\alpha<1$, we have $N_j2^{-\alpha j}=
o(j)$, so that 
\[
 P(X_j=0)=(1-2^{-\alpha j})^{N_j}\sim 1-N_j2^{-\alpha j}.
\]
Moreover 
\[
 P(X_k=1)=N_k2^{-\alpha k}(1-2^{-\alpha k})^{
N_k-1}\sim N_k2^{-\alpha k}. 
\]
Hence we get
\begin{align*}
 P(T_{\lambda}^{\eta,\alpha}\cap\Lambda=\{\lambda\})
&\sim\exp\Big( \sum_{j=k/\eta,j\neq k}^{j=\eta k} \ln (P(X_j=0)\Big)\times N_k2^{-\alpha k}\\
 &\sim \left(1-\sum_{j=k/\eta,j\neq k}^{j=\eta k}N_j2^{-\alpha j}\right)\times N_k2^{-\alpha k}.
\end{align*}
Again we use $\gamma<\alpha<1$ to see now that the sum $\sum_{j=k/\eta,j\neq k}^{\eta k}N_j2^{-\alpha j}$ is convergent and goes to zero when $k\to\infty$. This shows in particular that the 
fact of considering the event of having points in neighboring annuli of $A_k$ containing
$\lambda$ can be neglected.
Hence
\[
 P(T_{\lambda}^{\eta,\alpha}\cap\Lambda=\{\lambda\})\sim N_k2^{-\alpha k}.
\]
We now sum over all $\lambda\in\Lambda$ by summing over all dyadic annuli $A_k$ and the 
$N_k$ points contained in each annuli:
\[
 \sum_{\lambda\in\Lambda}P(T_{\lambda}^{\eta,\alpha}\cap\Lambda=\{\lambda\})
 \sim\sum_{k\in\NN} N_k\times N_k2^{-\alpha k}
 =\sum_{k\in\NN}N_k^22^{-\alpha k}.
\]
For $\alpha>\gamma$, this sum converges by assumption. 
Using the Borel-Cantelli lemma we deduce
that $T_{\lambda}^{\eta,\alpha}\cap\Lambda=\{\lambda\}$ for all but finitely many
$\lambda$ with probability one. Obviously these finitely many neighborhoods $T_{\lambda}^{\eta,
\alpha}$ contain finitely many points between which a lower Dirichlet distance exists.
This achieves the proof of the separation.
\end{proof}

It should be observed that the above proof only involves $\alpha$ but not $\eta$, so that
it is the separation in the annuli which dominates the situation.


\section{Proof of Theorem \ref{sintD}} 
\label{SectionIntD}

In order to prove Theorem  \ref{sintD} we will use a method similar to the one for $\cD_{\alpha}$.

Let us again observe that interpolation implies the zero sequence condition,
so that by Bogdan's result, if $\Lambda$ is almost surely interpolating then 
$\sum_{n}N_n/n<+\infty$.

So suppose now $\sum_{n}N_n/n<+\infty$. As mentioned 
in the proof of Theorem \ref{CoroThm5}, it is enough to check that we almost
surely have \eqref{1BoxSteg}, now with $\alpha=1$. So the condition translates to
\[
 \sum_{\lambda\in S_I}\left(\log\frac{e}{1-|\lambda|}\right)^{-1}
 \le c \left(\log \frac{1}{|I|}\right)^{-1},
\] 
which, using the usual dyadic discretization $I=I_{n,k}$, translates to 
\[
\sum_{m\geq n} \frac{1}{m}X_{n,m,k} \leq C \frac{1}{n} \,\,\, \text{almost surely}.
\]
This leads to estimate the tail of the random variables
\[
Y_{n,k}=\sum_{m\geq n}\Big(\frac{n}{m}\Big) X_{n,m,k}.
\]
To do that, introduce again two positive parameters $ s, A $. Using the formula for the generating function of a binomially distributed random variable and Markov's inequality we can estimate as follows 

\begin{align*}
\log P(Y_{n,k} \geq A ) &= \log P(s^{Y_{n,k}}>s^A) \\
&\leq \sum_{m\geq n} N_m \log(1-2^{-n}+2^{-n}s^{ (\frac{n}{m}) })-A\log(s) \\
&\leq 2^{-n}\sum_{m\geq n}\frac{N_m}{m} \big( \frac{m}{n} \big) (s^{ (\frac{n}{m}) }-1) n - A\log(s)\\
&\leq n 2^{-n} s \sum_{m\geq n} \frac{N_m}{m}-A\log(s).
\end{align*}

Setting  $ s=2^{n/2} $ and $ A=4 $ the above calculation gives 
\[
P(Y_{n,k}>4) \lesssim C 2^{-2n}.
\]
Again, an application of the Borel-Cantelli Lemma concludes the proof.

As already mentioned in preceding cases, the Kolmogorov 0-1 law allows to get
also the cases of a divergent sum and interpolation with probability 0.

\section{Annex : Proof of Theorem \ref{alphazero} }\label{A}

Carleson proved in \cite[Theorem 2.2]{Ca} that,  for $0<\alpha<1$, if
\[
\sum_{\lambda\in \Lambda} (1-|\lambda|)^{1-\alpha}<\infty
\]
then the Blaschke product $B$ associated  to $\Lambda$  belongs to $\cD_\alpha$.
So the sufficiency part of Theorem \ref{alphazero} follows immediately from this result
(and is moreover deterministically true).
\\

For the proof of the converse  
we will need the following two lemmas. 
The first one is a version of the Borel-Cantelli Lemma \cite[Theorem 6.3]{Bil}.
\begin{lem}\label{BoCanV2}If  $\{A_n\}$ is a sequence of measurable subsets in a probability space $(X,P)$ such that $\sum P(A_n) =\infty$ and
\begin{equation}\label{bocanV1}
\liminf_{n\to \infty}\frac{\sum_{j, k\leq n}P(A_j\cap A_k)}{[\sum_{k\leq n}P(A_k)]^2}\leq 1,
\end{equation}
then $P(\limsup_{n\to \infty} A_n)=1$.
\end{lem}
The second Lemma is due to Nagel, Rudin and Shapiro \cite{NRS, NRS2} who discussed tangential approach regions of functions in $\cD_\alpha$.
\begin{lem}\label{NRST}Let $f\in \cD_\alpha$, $0<\alpha<1$. Then, for a.e. $\zeta\in \TT$, we have $f(z)\to f^*(\zeta)$ as $z\to \zeta$ in each region
\[
 |z-\zeta|<\kappa (1-|z|)^{1-\alpha},\qquad (\kappa>1).
\]
\end{lem}
\begin{proof}[Proof of Theorem \ref{alphazero}]
In view of our preliminary observations, we are essentially interested in the converse
implication.
So suppose $ \sum_n 2^{-(1-\alpha)n}N_n=+\infty$ or equivalently
\begin{equation}\label{diverC}
  \sum_n (1-\rho_n)^{1-\alpha}=+\infty.
\end{equation}
We have to show that $\Lambda$ is 
not a zero sequence almost surely. For this,
introduce the intervals $I_{\ell}=({\rm e}^{-i(1-\rho_{\ell})^{1-\alpha}}, {\rm e}^{i(1-\rho_{\ell})^{1-\alpha}})$ 
and let $F_{\ell}={\rm e}^{i\theta_{\ell}}I_{\ell}$.
Denoting by $m$ normalized Lebesgue measure on $\TT$,
observe that 
\[
 m(F_{\ell})=m(I_{\ell})=(1-\rho_{\ell})^{1-\alpha}.
\]
We have for every $\varphi\in F_{\ell}$,
$\lambda_{\ell} \in \Omega_{\kappa,\varphi}=\{z\in \DD : |z-{\rm e}^{i\varphi}|<\kappa (1-|z|)^{1-\alpha}\}.$ By Lemma \ref{NRST} it suffices to prove that $|\limsup_{\ell}F_{\ell}|>0$ a.s.\ (the latter condition
means that there is a set of strictly positive measure on $\TT$ to which $\Lambda$ accumulates in Dirichlet  tangential approach regions according to Lemma \ref{NRST}, which is of course not
possible for a zero sequence). 
Let $E$ denote the expectation with respect to the Steinhaus sequence $(\theta_n)$.
By Fubini's theorem we have $E[m(F_j\cap F_k)]=m(I_j)m(I_k)$, $j\neq k$, (the expected size
of intersection of two intervals only depends on the product of the length  of both intervals). By Fatou's Lemma and \eqref{diverC}
\begin{align*}
E\Big[\liminf_{n\to \infty}\frac{\sum_{j, k\leq n}m(F_j\cap F_k)}{[\sum_{k\leq n}m(F_k)]^2}\Big]&\leq \liminf_{n\to \infty}E\Big[\frac{\sum_{j, k\leq n}m(F_j\cap F_k)}{[\sum_{k\leq n}m(F_k)]^2}
 \Big]\\
&=\liminf_{n\to \infty}\frac{\sum_{j, k\leq n}E[m(F_j\cap F_k)]}{[\sum_{k\leq n}m(F_k)]^2}\\
&=\liminf_{n\to \infty}\frac{\sum_{j,k\leq n, j\neq k}m(I_j)m(I_k)+\sum_{k\leq n} m(I_k)}{[\sum_{k\leq n}m(I_k)]^2}\\
&=\liminf_{n\to \infty}\left(1+\frac{\sum_{k\leq n} m(I_k)(1-m(I_k))}{[\sum_{k\leq n}m(I_k)]^2}
 \right).
\end{align*}
Now, since $1-m(I_k)\to 1$, and by \eqref{diverC}, keeping in mind that $m(I_k)=(1-\rho_k)^{1-\alpha}$, we have
\[
 \lim_{n\to\infty} \frac{\sum_{k\leq n} m(I_k)(1-m(I_k))}{[\sum_{k\leq n}m(I_k)]^2}=0.
\]
This implies that \eqref{bocanV1} holds on a set $B$ of positive probability and hence, by the
zero-one law, on a set of probability one.
From Lemma \ref{BoCanV2} we conclude $P(\limsup_{n\to \infty} F_n)=1$ a.s., which is what
we had to show.
 \end{proof}

\end{document}